\documentclass[12pt,a4paper]{amsart}
\usepackage{amssymb,amscd}
\usepackage[all]{xy}
\usepackage{graphicx}

\theoremstyle{plain}
\newtheorem{thm}{Theorem}[section]
\newtheorem{lem}[thm]{Lemma}
\newtheorem{cor}[thm]{Corollary}

\newcommand{\Z}{\mathbb{Z}}
\newcommand{\N}{\mathbb{N}}

\newcommand{\FF}{\mathbb{F}}

\DeclareMathOperator{\Hom}{Hom}
\DeclareMathOperator{\End}{End}
\DeclareMathOperator{\Ext}{Ext}
\DeclareMathOperator{\Ker}{Ker}
\DeclareMathOperator{\Ima}{Im}
\DeclareMathOperator{\Def}{Def}
\DeclareMathOperator{\Ind}{Ind}
\DeclareMathOperator{\Coker}{Coker}
\DeclareMathOperator{\Rep}{Rep}

\begin{document}
\title[]{Representations of equipped graphs: Auslander-Reiten theory}

\author{William Crawley-Boevey}
\address{Fakult\"at f\"ur Mathematik, Universit\"at Bielefeld, Postfach 100131, 33501 Bielefeld, Germany}
\email{wcrawley@math.uni-bielefeld.de}

\thanks{The author is supported by the Alexander von Humboldt Foundation in the framework of an 
Alexander von Humboldt Professorship endowed by the German Federal Ministry of Education and Research.}
\thanks{This paper is in final form and no version of it will be submitted for publication elsewhere}

\begin{abstract}
Representations of equipped graphs were introduced by Gelfand and Ponomarev; 
they are similar to representation of quivers, but one does not need to choose an orientation of the graph.
In a previous article we have shown that, as in Kac's Theorem for quivers,
the dimension vectors of indecomposable representations are exactly the positive
roots for the graph. In this article we begin by surveying that work, and then we go on to discuss 
Auslander-Reiten theory for equipped graphs, and give examples of Auslander-Reiten quivers.
\end{abstract}
\maketitle
\thispagestyle{empty}
\pagestyle{empty}

\section{Equipped graphs}
Recall that a \emph{quiver} is a tuple $Q = (Q_0,Q_1,h,t)$ where $Q_0$ is a set of \emph{vertices}, $Q_1$ is a set of \emph{arrows}, 
and $h,t:Q_1\to Q_0$ assign a \emph{head} $h(a)$ and \emph{tail} $t(a)$ for each arrow $a$, so $a:t(a)\to h(a)$.
By definition a \emph{graph} is a pair $(G,*)$ where $G$ is a quiver and $*$ is a fixed-point-free involution on $G_1$, 
exchanging heads and tails, that is $(a^*)^* = a$, $a^* \neq a$ and $h(a^*) = t(a)$ for $a\in G_1$.
(This definition appears for example in \cite[\S2.1]{S}. It allows a graph to have multiple distinguishable edges between two vertices, 
and also to have edge-loops.)

One can pass between quivers and graphs as follows.
The \emph{underlying graph} of a quiver $Q$ is the the graph $(\overline{Q},*)$ where $\overline{Q}$ is the \emph{double} of $Q$,
obtained from $Q$ by adjoining a new arrow $a^*:j\to i$ for each arrow $a:i\to j$ in $Q$, and $*$ is the involution exchanging $a$ and $a^*$.
Conversely to pass from a graph $G$ to a quiver $Q$ one chooses an \emph{orientation} of $G$, which is a subset $\Omega$ of $G_1$
which meets each $G$-orbit in exactly one point. The quiver is then $(G_0,\Omega,h|_\Omega,t|_\Omega)$.

With this setup, the definition of an \emph{equipped graph} $E = ((G,*),\Omega)$ 
becomes very natural: it consists of a graph $(G,*)$ and an arbitrary subset $\Omega$ of $G_1$.

The set $\Omega$ has a characteristic function $\phi:G_1\to \{0,1\}$, taking the value 1 on elements in $\Omega$,
and 0 on elements not in $\Omega$, and instead of fixing $\Omega$, it is equivalent to fix a function $\phi$.

We picture an equipped graph $((G,*),\Omega)$ as follows. We draw a vertex for each element of $G_0$, and
for each orbit of $*$ on $G_1$, say $\{a,a^*\}$ with $a:i\to j$ and so $a^*:j\to i$, we draw a decorated edge as follows
\[
\begin{array}{c|cc}
& a\in\Omega  &  a\notin\Omega
\\
\hline \\[-8pt]
a^*\in\Omega &
\stackrel{i}{\bullet} -\!\!\!-\!\!\!- \stackrel{j}{\bullet}
&
\stackrel{i}{\bullet} \ \longleftarrow \ \stackrel{j}{\bullet}
\\[3pt]
a^*\notin\Omega &
\stackrel{i}{\bullet} \ \longrightarrow \ \stackrel{j}{\bullet}
&
\stackrel{i}{\bullet} \ \leftarrow\!\rightarrow \ \stackrel{j}{\bullet}
\end{array}
\]
Thus $\Omega$ corresponds to the set of tails in the picture. (Thinking just about tails, one could instead define an equipped graph
as a tuple $(G_0,G_1,t,*,\Omega)$ where $G_0$ and $G_1$ are sets, $t:G_1\to G_0$, $*$ is a fixed-point-free involution on $G_1$ and $\Omega$
is a subset of $G_1$; one recovers $h$ as $t\circ *$.)

\section{Relations}
Let $K$ be a field.
Given vector spaces $V$ and $W$, a \emph{(linear) relation} from $V$ to $W$ is a subspace $R\subseteq V\oplus W$.
The graph of a linear map $f:V\to W$ is a relation $\{ (v,f(v)) : v\in V\}$, and in general we think of a relation
as a generalization of a linear map.
In this spirit, the composition of relations $R\subseteq V\oplus W$ and $S\subseteq U\oplus V$ is 
$RS = \{ (u,w) \in U\oplus W : \text{$(u,v)\in S$ and $(v,w)\in R$ for some $v\in V$}\}$.
Any relation $R$ gives a subspace $\{ (w,v) \colon (v,w)\in R\}$ of $W\oplus V$; we call it the 
\emph{opposite} $R^{op}$ or \emph{inverse} $R^{-1}$ of $R$.

Gelfand and Ponomarev \cite{GP} observed that a relation $R\subseteq V\oplus W$ induces 
subspaces of $\Ker R \subseteq \Def R \subseteq V$ and $\Ind R \subseteq \Ima R \subseteq W$,
\[
\begin{array}{ll}
\Ker R = \{ v\in V \colon (v,0)\in R \} & \text{Kernel} \\
\Def R = \{ v\in V \colon \text{$(v,w)\in R$, some $w\in W$} \} & \text{Domain of definition} \\
\Ind R = \{ w\in W : (0,w)\in R \} & \text{Indeterminacy} \\
\Ima R = \{ w\in W : \text{$(v,w)\in R$, some $v\in V$} \}  & \text{Image.}
\end{array}
\]
Moreover they observed that $R$ induces an isomorphism between suitable quotient spaces.
We call this the \emph{Isomorphism Theorem}. The proof is straightforward.

\begin{thm}
Any relation $R\subseteq V\oplus W$ induces a linear map from $\Def R$ to $W/\Ind R$ and an isomorphism $\Def R/\Ker R\cong \Ima R/\Ind R$.
Any isomorphism between a subquotient of\/ $V$ and a subquotient of\/ $W$ arises in this way 
from some relation $R$, and together with the subspaces, it uniquely determines $R$.
\end{thm}

Linear relations were studied by Maclane \cite{M}, evidently with the intention that they should play a wider role in homological algebra.
As an illustrative example, given a commutative diagram with exact rows
\[
\begin{CD}
0 @>>> X @>f>> Y @>g>> Z @>>> 0 \\
& & @V\theta VV @V\phi VV @V\psi VV \\
0 @>>> X' @>f'>> Y' @>g'>> Z' @>>> 0
\end{CD}
\]
\\
the connecting map $\Ker \psi\to \Coker\theta$ in the Snake Lemma is induced by the relation $(f')^{-1}\phi g^{-1}$.

Following Gelfand and Ponomarev, a relation $R\subseteq V\oplus W$ is said to be 
an \emph{equipped relation of type $(r,s)$} with $r,s\in\{0,1\}$,
provided that: if $r=1$ then $\Def R=V$, if $r=0$ then $\Ker R=0$, if $s=1$ then $\Ima R = W$, if $s=0$ then $\Ind R = 0$.

We indicate that $R\subseteq V\oplus W$ is an equipped relation of type $(r,s)$ by drawing an edge as follows
\[
\begin{array}{c|cc}
& r=1  &  r=0
\\
\hline \\[-8pt]
s=1 &
V -\!\!\!-\!\!\!- W
&
V \longleftarrow W
\\[3pt]
s=0 &
V \longrightarrow W
&
V \leftarrow\!\rightarrow W
\end{array}
\]
An equipped relation of type $(1,0)$ is (the graph of) a linear map, and one of type $(0,1)$ is the inverse of a linear map.
An equipped relation of type $(1,1)$ is determined by the subspaces $\Ker R\subseteq V$
and $\Ind R\subseteq W$ and an isomorphism $V/\Ker R \cong W/\Ind R$, while
an equipped relation of type $(0,0)$ is determined by the subspaces $\Def R\subseteq V$
and $\Ima R\subseteq W$ and an isomorphism $\Def R \cong \Ima R$.

\section{Representation of equipped graphs}
The category $\Rep E$ of representations of a (finite) equipped graph $E= ((G,*),\Omega)$ is defined, following \cite{GP}, as follows.

An object $X$ consists of a (finite-dimensional) vector space $X_i$ for each vertex $i\in G_0$ and a relation $X_a \subseteq X_{t(a)}\oplus X_{h(a)}$
for each arrow $a\in G_1$, which is equipped of type $(\phi(a),\phi(a^*))$ (where $\phi$ is the characteristic function of $\Omega$),
and with $X_{a^*} = (X_a)^{-1}$.

A morphism $\theta:X\to Y$ consists of a linear map $\theta_i:X_i\to Y_i$ for each $i\in G_0$ such that
$(\theta_{t(a)}(v),\theta_{h(a)}(w)) \in Y_a$ for all $a\in G_1$ and $(v,w)\in X_a$.

It is clear that $\Rep E$ is an additive $K$-category: the direct sum $X\oplus Y$ of two representations is given by 
$(X\oplus Y)_i = X_i\oplus Y_i$ for $i\in G_0$ and $(X\oplus Y)_a = X_a\oplus Y_a$ for $a\in G_1$.
The \emph{dimension vector} of a 
representation $X$ is the vector $\underline\dim X \in \N^{G_0}$ whose $i$th component is $\dim X_i$.

In \cite{CB} we proved the following version of Kac's Theorem \cite{Kac1,Kac2}. The root
system is the same as for Kac. In case $G$ has no loops, it is the root system for a Kac-Moody Lie algebra.

\begin{thm}
\label{t:eqkac}
Over an algebraically closed field $K$, the dimension vectors of indecomposable representations of $E$ are exactly the
positive roots for the underlying graph $G$. Up to isomorphism there is a unique indecomposable for each positive real root, 
infinitely many for each positive imaginary root.
\end{thm}

In particular, for an algebraically closed field this recovers the result of Gelfand and Ponomarev \cite{GP}, that an 
equipped graph has only finitely many indecomposable representations if and only if the graph is Dynkin (if connected),
and in this case the indecomposable representations correspond to positive roots.

\section{(Elementary) deduction of Kac's Theorem}
Our version of Kac's Theorem is deduced from the original version for quivers. 
To begin, we observe that the category of representation of an equipped graph can be
embedded in the category of representations of a suitable quiver.

Let $Q$ be a finite quiver and let $I,S\subseteq Q_1$. We write $\Rep^{I,S} Q$ for the category of
finite-dimensional representations of $Q$ in which the arrows in $I$ are injective and the arrows in $S$ are surjective.
This is a full subcategory of $\Rep Q$, which is clearly closed under extensions and direct summands.

\begin{lem}
\label{l:assocq}
Given an equipped graph $E$, there is an associated quiver $Q$
and subsets $I,S \subseteq Q_1$, such $\Rep E$ is equivalent to $\Rep^{I,S} Q$.
\end{lem}

For example, for the equipped graph
\[
\stackrel{1}{\bullet} -\!\!\!-\!\!\!- 
\stackrel{2}{\bullet} \ \leftarrow\!\rightarrow \ 
\stackrel{3}{\bullet} \ \longrightarrow \ 
\stackrel{4}{\bullet}
\]
one can take $Q$ to be the following quiver 
(where the arrows in $I$ are drawn with a hook, and the arrows in $S$ are drawn with a double head)
\[
\stackrel{1}{\bullet} \twoheadrightarrow \bullet \twoheadleftarrow
\stackrel{2}{\bullet} \hookleftarrow \bullet \hookrightarrow
\stackrel{3}{\bullet} \ \longrightarrow \ 
\stackrel{4}{\bullet}
\]
or the quiver
\[
\stackrel{1}{\bullet} \twoheadrightarrow \bullet \twoheadleftarrow
\stackrel{2}{\bullet} \hookleftarrow \bullet \hookrightarrow
\stackrel{3}{\bullet} \twoheadrightarrow \bullet \hookrightarrow
\stackrel{4}{\bullet}.
\]
In general one needs to replace each double tailed edge 
$\bullet -\!\!\!-\!\!\!- \bullet$ by 
$\bullet \twoheadrightarrow \bullet \twoheadleftarrow \bullet$,
each double headed edge
$\bullet \leftarrow\!\rightarrow \bullet$
by $\bullet \hookleftarrow \bullet \hookrightarrow \bullet$,
and an ordinary arrow $\bullet \longrightarrow \bullet$ can be left as it is, or replaced
with $\bullet \twoheadrightarrow \bullet \hookrightarrow \bullet$.
The proof of the lemma is straightforward, using the Isomorphism Theorem, in the form stated here
and the usual version, that any linear map can be factorized as a surjection followed by an injection.

If we leave each ordinary arrow as it is, we obtain the associated quiver with the minimal number of vertices;
if we replace each ordinary arrow with a surjection followed by an injection, we obtain the associated quiver with
the maximal number of vertices. Both versions have their uses.

The proof of Kac's Theorem for quivers \cite{Kac1,Kac2} involves a study of the number
$A_{Q,\alpha}(q)$ of absolutely indecomposable representations of $Q$ of dimension $\alpha$ over 
a finite field $\FF_q$. Recall that a representation is \emph{absolutely indecomposable} provided it
is indecomposable, and remains so over the algebraic closure of the base field.

The following four properties are proved by Kac. Here 
\[
q(\alpha) = \sum_{i\in G_0} \alpha_i^2 - \frac12 \sum_{a\in G_1} \alpha_{t(a)}\alpha_{h(a)}
\]
is the quadratic form on $\N^{G_0}$ given by the underlying graph $G$ of $Q$.

\begin{lem}\ 
\label{l:kaclemma}
\begin{itemize}
\item[(i)]$A_{Q,\alpha}(q)\in \Z[q]$.
\item[(ii)]$A_{Q,\alpha}(q)\neq 0 \Leftrightarrow$ $\alpha$ is a positive root.
\item[(iii)]If so, then $A_{Q,\alpha}(q)$ is a monic polynomial of degree $1 - q(\alpha)$.
\item[(iv)]$A_{Q,\alpha}(q)$ depends on $Q$ only through its underlying graph.
\end{itemize}
\end{lem}

Given a quiver and a subset $M \subseteq Q_1$ we define $A^M_{Q,\alpha}(q)$ to be the
number of absolutely indecomposable representations of $Q$ of dimension $\alpha$ over 
$\FF_q$ in which the linear maps corresponding to arrows in $M$ have maximal rank.
Recall that this means that they are either injective or surjective - which applies is determined
by the dimension vector of the representation.
We obtain the following, which is \cite[Theorem 2.1]{CB}.

\begin{lem}\ 
\label{l:kacmax}
\begin{itemize}
\item[(i)]$A^M_{Q,\alpha}(q)\in \Z[q]$.
\item[(ii)]$A^M_{Q,\alpha}(q)\neq 0 \Leftrightarrow$ $\alpha$ is a positive root.
\item[(iii)]If so, then $A^M_{Q,\alpha}(q)$ is a monic polynomial of degree $1 - q(\alpha)$.
\item[(iv)]$A^M_{Q,\alpha}(q)$ depends on $Q$ only through its underlying graph.
\end{itemize}
\end{lem}

The idea of the proof is that if one wants to count representations in which an arrow $a:i\to j$ has maximal rank,
one can count all representations and subtract the number of representations in which it has rank
$r < \min(\alpha_i,\alpha_j)$. 
Now if the map representing $a$ has rank $R$, then it factorizes through a vector space of dimension $r$.
Thus if $a\notin M$ then
\[
A^{M\cup\{a\}}_{Q,\alpha}(q) = A^M_{Q,\alpha}(q) - \sum_{r<  \min(\alpha_i,\alpha_j)} A^{M\cup\{b,c\}}_{Q^r,\alpha^r}(q)
\]
where $Q^r$ is the quiver obtained from $Q$ by replacing $a$ by 
\[
\stackrel{i}\bullet \ \stackrel{b}{\twoheadrightarrow} \ \stackrel{k}{\bullet} \ \stackrel{c}{\hookrightarrow} \ \stackrel{j}{\bullet}
\]
and $\alpha^r$ is the dimension vector with $\alpha_k = r$ and otherwise equal to $\alpha$.
One can then use induction.

For an equipped graph $E$ we can equally well define $A_{E,\alpha}(q)$ to be the number of absolutely indecomposable
representations of $E$ of dimension vector $\alpha$, and we have the following.

\begin{lem}\ 
\label{l:eqlemma}
\begin{itemize}
\item[(i)]$A_{E,\alpha}(q)\in \Z[q]$.
\item[(ii)]$A_{E,\alpha}(q)\neq 0 \Leftrightarrow$ $\alpha$ is a positive root.
\item[(iii)]If so, then $A_{E,\alpha}(q)$ is a monic polynomial of degree $1 - q(\alpha)$.
\item[(iv)]$A_{E,\alpha}(q)$ depends on $E$ only through its underlying graph $G$.
\end{itemize}
\end{lem}

Indeed $A_{E,\alpha}$ is a sum of various $A^M_{Q,\beta}$, where $Q$ is the quiver corresponding to $E$ in Lemma~\ref{l:assocq}
(with the maximal number of vertices),
$M = I\cup S$, and $\beta$ runs through the finite number of possible dimension vectors of $Q$ which are compatible with
the linear maps in $I$ and $S$ being injective and surjective, and whose restriction to the vertices in $E$ is equal to $\alpha$.
Now (iv) follows from Lemma~\ref{l:kacmax}(iv). But then $A_{E,\alpha}(q)$ is the same as for a quiver with the same underlying graph as $E$,
and hence (i), (ii), (iii) follow from Lemma~\ref{l:kaclemma}.

Finally Kac's Theorem for quivers follows from Lemma~\ref{l:kaclemma} by certain arguments passing between algebraically
closed fields and finite fields. Exactly the same arguments can be used for equipped graphs, to deduce Theorem~\ref{t:eqkac}
from Lemma~\ref{l:eqlemma}.

\section{Auslander-Reiten Theory}

Let $Q$ be a finite quiver and let $I$ and $S$ be disjoint subsets of $Q_1$. 
By an \emph{oriented $A$-cycle}, with $A\subseteq Q_1$,  we mean a word of the form $w_0 w_1 \dots w_{n-1}$ 
with $n\ge 1$, where subscripts are taken modulo $n$,
each letter $w_i$ is either in $Q_1$ or is a formal inverse $a^{-1}$ with $a\in A$, 
consecutive letters are not inverses of each other,
and $t(a_i) = h(a_{i+1})$ for all~$i$, with the convention that $h(a^{-1})=t(a)$ and $t(a^{-1})=h(a)$.

\begin{thm}
If $Q$ has no oriented $I$- or $S$-cycles,
then $\Rep^{I,S}Q$ is functorially finite in $\Rep Q$.
\end{thm}

\begin{proof}
We show it is covariantly finite, that is, for every $KQ$-module $X$ there is a morphism $f:X\to Y$ with $Y$ in
$\Rep^{I,S}Q$, and such that any morphism from $X$ to an object in $\Rep^{I,S}Q$ factors through $f$.
By duality one gets contravariant finiteness, so  $\Rep^{I,S}Q$ is functorially finite in $\Rep Q$.

Given an arrow $a$, the $K Q$-module $C(a) = K Q e_{t(a)} / K Q a$ fits in an exact sequence
$
0 \to KQ e_{h(a)}\to KQ e_{t(a)} \to C(a) \to 0
$
where the first map is right multiplication by $a$.
Applying $\Hom(-,X)$ to this sequence gives an exact sequence
\[
0 \to \Hom(C(a),X) \to e_{t(a)} X \xrightarrow{X_a} e_{h(a)} X \to \Ext^1(C(a),X) \to 0,
\]
so $X_a$ is injective if and only if $\Hom(C(a),X)=0$ and $X_a$ is surjective if and only if $\Ext^1(C(a),X)=0$.

If $a\neq b$ then $\Hom(C(a),C(b))=0$, for if $\theta$ is such a morphism, then 
$\theta (KQ a + e_{t(a)}) = K Q b + z$ for some $z \in K Q e_{t(b)}$, and we may assume that $z \in e_{t(a)} K Q e_{t(b)}$.
Now $a$ times the element $KQ a + e_{t(a)}$ is zero, so $a$ times the element $K Q b + z$, must be zero, so $a z\in K Q b$.
But since $a\neq b$ this implies that $z \in K Q b$. (Similarly $\End(C(a)) = K$.)

Since there are no oriented $I$-cycles, we can write $I = \{ a_1,a_2,\dots,a_s\}$ 
such that if $i\le j$ then any path from $t(a_j)$ to $h(a_i)$ is of the form $q a_j$ or $a_i q$ for some path $q$.
We show that $\Ext^1(C(a_i),C(a_j))=0$ for $i\le j$.
This is the condition that the arrow $a_i$ in $C(a_j)$ is surjective.
Thus that 
\[
e_{h(a_i)} ( K Q e_{t(a_j)} / K Q a_j ) = a_i ( K Q e_{t(a_j)} / K Q a_j ).
\]
This holds by the condition on paths from $t(a_j)$ to $h(a_i)$.
Similarly, by the condition on $S$-cycles, we can write $S = \{ b_1,\dots,b_t\}$ with $\Ext^1(C(b_i),C(b_j))=0$ for $i\le j$.

Given a representation $X$ of $Q$, we construct morphisms
\[
X = X^0 \to X^1 \to \dots \to X^s
\]
such that in the representation $X^j$ the maps $a_1,\dots,a_j$ are injective, and such that any map from $X$ to a representation
in $\Rep^{I,S}Q$ factors through $X^j$.
Suppose we have constructed $X^{j-1}$. We take $X^{j-1}\to X^j$ to be the cokernel of the universal map
$\phi: C(a_j)^N\to X^{j-1}$ where $N = \dim\Hom(C(a_j),X^{j-1})$.
Any map $X^{j-1}\to Y$ with $Y$ in $\Rep^{I,S}Q$ factors through $X^j$ since $\Hom(C(a_j),Y)=0$.
For $i\le j$ we have exact sequences
\[
\begin{split}
\Hom(C(a_i),\Ima \phi) &\to \Hom(C(a_i),X^{j-1}) \to \Hom(C(a_i),X^j) \to \\
& \to \Ext^1(C(a_i),\Ima\phi)
\end{split}
\]
and
\[
\begin{split}
\Hom(C(a_i),C(a_j)^N) &\to \Hom(C(a_i),\Ima\phi) \to \Ext^1(C(a_i),\Ker\phi) \to \\
&\to \Ext^1(C(a_i),C(a_j)^N)\to \Ext^1(C(a_i),\Ima\phi)\to 0
\end{split}
\]
Now $\Ext^1(C(a_i),C(a_j)) = 0$, so $\Ext^1(C(a_i),\Ima\phi)=0$.
If $i\le j-1$ then $\Hom(C(a_i),X^{j-1})=0$., so $\Hom(C(a_i),X^j)=0$, so $X^j$ has $a_i$ injective.
By the universal property of $\phi$, the map
\[
\Hom(C(a_j),C(a_j)^N) \to  \Hom(C(a_j),X^{j-1})
\]
is onto, hence so is $\Hom(C(a_j),\Ima\phi) \to \Hom(C(a_j),X^{j-1})$. 
It follows that $\Hom(C(a_j),X^j)=0$, so $X^j$ has $a_j$ injective.

We now construct morphisms
\[
X \to X^s = Z^0 \to Z^1 \to \dots \to Z^t
\]
such that in the representation $Z^j$ the maps in $I$ are injective and $b_1,\dots,b_j$ are surjective, and such that any map from $X$ to a representation
in $\Rep^{I,S}Q$ factors through $Z^j$.
Suppose we have constructed $Z^{j-1}$. We take $Z^{j-1}\to Z^j$ to be the universal extension
\[
0\to Z^{j-1} \to Z^j \to C(b_j)^N\to 0
\]
where $N = \dim \Ext^1(C(b_j),Z^{j-1})$. 
Any map $g:Z^{j-1}\to Y$ with $Y$ in $\Rep^{I,S}Q$ factors through $Z^j$ since $\Ext^1(C(b_j),Y)=0$,
so the pushout of this exact sequence along $g$ splits.
For any $i$ we have an exact sequence
\[
0 \to \Hom(C(a_i),Z^{j-1}) \to \Hom(C(a_i),Z^j) \to \Hom(C(a_i),C(b_j)^N)
\]
and since $I$ and $S$ are disjoint we have $\Hom(C(a_i),C(b_j))=0$, so $a_i$ is injective in $Z^j$.
For $i\le j-1$ we have exact sequences
\[
\Ext^1(C(b_i),Z^{j-1}) \to \Ext^1(C(b_i),Z^j) \to \Ext^1(C(b_i),C(b_j)^N).
\]
Since $\Ext^1(C(b_i),C(b_j))=0$ and $\Ext^1(C(b_i),Z^{j-1})=0$, we deduce that $ \Ext^1(C(b_i),Z^j)=0$,
so $b_i$ is a surjection in $Z^j$.
We have a long exact sequence
\[
\begin{split}
\Hom(C(b_j),C(b_j)^N) &\xrightarrow{c} \Ext^1(C(b_j),Z^{j-1}) \to \Ext^1(C(b_j),Z^j) \to \\
&\to \Ext^1(C(b_j),C(b_j)^N).
\end{split}
\]
By the universal property the connecting map $c$ is surjective, 
and also $\Ext^1(C(b_j),C(b_j))=0$. 
Thus $\Ext^1(C(b_j),Z^j)=0$, so $b_j$ is a surjection in the representation $Z^j$.
Now $X\to Z^t$ is the wanted morphism.
\end{proof}

By results of Auslander and Smal{\o} \cite{AS}, 
under the hypotheses of the theorem, the category $\Rep^{I,S} Q$ has Auslander-Reiten sequences.

To apply this to an equipped graph $E$ we use Lemma~\ref{l:assocq}.
There is a natural exact structure on $\Rep E$, with 
$0\to X\to Y\to Z\to 0$
an exact sequence if and only if for each vertex $i\in G_0$ the sequence on vector spaces
$0\to X_i\to Y_i\to Z_i\to 0$ is exact, and for each arrow $a\in G_1$ the sequence on relations
$0\to X_a \to Y_a \to Z_a\to 0$ is exact.
This corresponds to the embedding of $\Rep E$ in $\Rep Q$ where $Q$ is the associated quiver 
with the minimal number of vertices.
We immediately obtain the following.

\begin{cor}
Let $E$ be an equipped graph, 
and suppose that there are no oriented cycles in $G$ on which the characteristic function
$\phi$ is constant.
Then $\Rep E$ has Auslander-Reiten sequences.
\end{cor}

\section{Examples of Auslander-Reiten quivers}

Let $E = ((G,*),\Omega)$ be an equipped graph. 
We restrict to the case $\Omega = \varnothing$, so that all arrows are two-headed,
and in order to ensure that we have Auslander-Reiten sequences we assume that the underlying graph is a tree.
In this case the associated quiver $Q$ in Lemma~\ref{l:assocq} is uniquely determined, $I = Q_1$, $S = \varnothing$,
and we identify $\Rep E$ with $\Rep^{I,S} Q$.

The projective modules in $\Rep Q$ are in $\Rep E$, and they are exactly the relative projectives.
The corresponding sink maps (i.e. minimal right almost split maps) are the same as for $Q$, so of the form
\[
\bigoplus_{t(a)=i} KQ e_{h(a)} \to KQ e_i.
\]
One uses knitting to compute the rest of the preprojective component. We write $\tau$ and $\tau^-$
for the Auslander-Reiten translations for $\Rep E$ (see the discussion before \cite[\S 2.3 Lemma 2]{R}),
and draw a dotted line to show the action of $\tau$.
Iteratively, one only draws a vertex for a representation
$X$ once one knows its dimension vector and can draw arrows for all irreducible maps ending at $X$.
Once one has drawn a representation $X$ and all irreducible map starting at $X$,
then by \cite[\S 2.2 Lemma 3]{R} one knows the dimension of the source map starting at $X$,
and hence by \cite[\S 2.3 Lemma 2]{R} one can draw $\tau^- X$.
As an example, if $E$ is the equipped graph $A_4$
\[
\stackrel{1}{\bullet} \ \leftarrow\!\rightarrow \ 
\stackrel{2}{\bullet} \ \leftarrow\!\rightarrow \ 
\stackrel{3}{\bullet} \ \leftarrow\!\rightarrow \ 
\stackrel{4}{\bullet}
\]
then the associated quiver $Q$ is
\[
\stackrel{1}{\bullet} \hookleftarrow \stackrel{1'}{\bullet} \hookrightarrow
\stackrel{2}{\bullet} \hookleftarrow \stackrel{2'}{\bullet} \hookrightarrow
\stackrel{3}{\bullet} \hookleftarrow \stackrel{3'}{\bullet} \hookrightarrow
\stackrel{4}{\bullet}
\]
and the Auslander-Reiten quiver is as in Figure~\ref{f:Afour}. In this case the representations of $Q$ are known;
we write $ij$ for the indecomposable representation which is 1-dimensional at vertices $i$ and $j$ and all vertices
between them, and zero at the remaining vertices.
Note that in the process of knitting one would expect an Auslander-Reiten sequence with left hand term $11$ and middle term $12$,
but then the right hand term would be $1'2$, which is not in the category $\Rep^{I,S}Q$, which shows that $11$ is a relative injective.

\begin{figure}
\[
{\xymatrix@=1.5pc{
11 \ar[dr]  \\
& 12 	\ar[dr] \\
22 \ar[ur]\ar[dr] & & 13 \ar[dr] \ar@{.}[ll] \\
& 23 \ar[ur]\ar[dr] & & 14  \ar@{.}[ll] \\
33 \ar[ur]\ar[dr] & & 24 \ar[ur]  \ar@{.}[ll] \\
& 34 \ar[ur] \\
44  \ar[ur]
}}
\]
\caption{Auslander-Reiten quiver for $A_4$}
\label{f:Afour}
\end{figure}
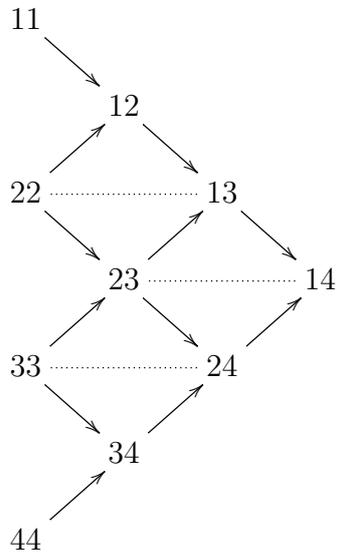

\begin{figure}
\[
\includegraphics[width=\textwidth]{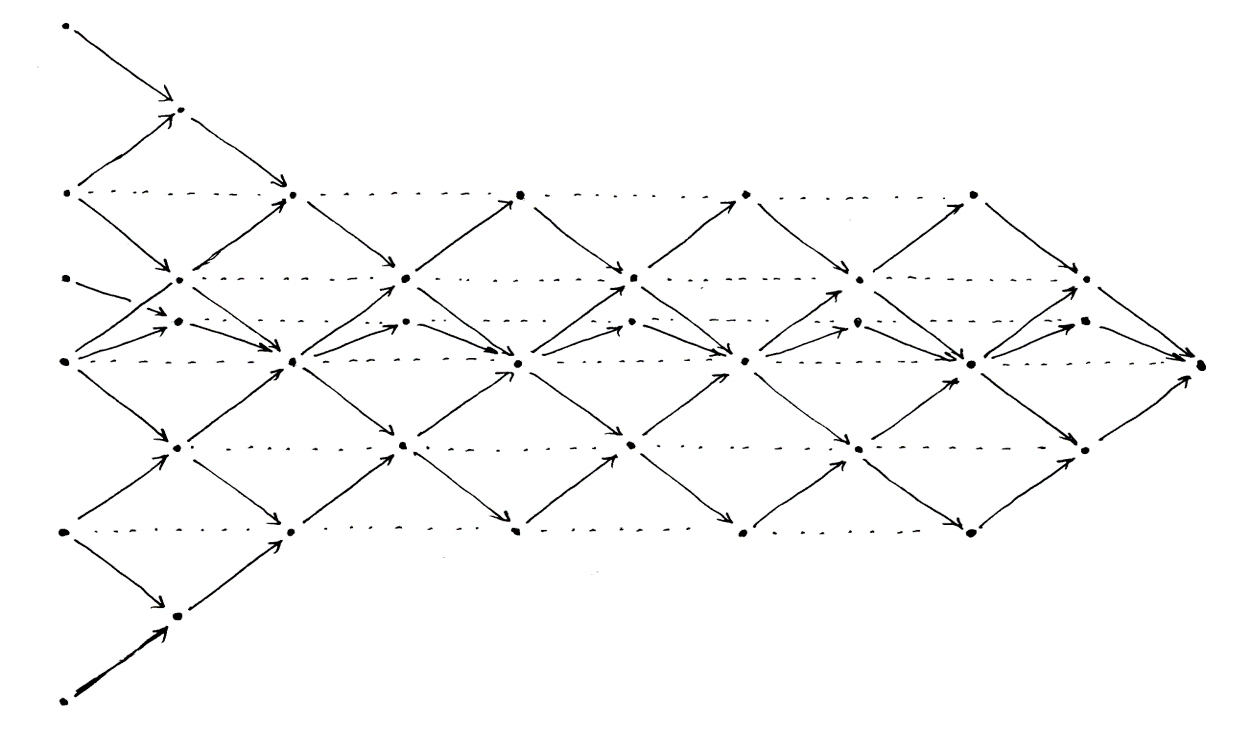}
\]
\vspace{-1cm}
\caption{Auslander-Reiten quiver for $E_6$}
\label{f:Esix}
\end{figure}

In Figure~\ref{f:Esix} we show the Auslander-Reiten quiver of the equipped graph~$E_6$
\[
{\xymatrix@=1.5pc{
& & \bullet \ar@{<->}[d]  \\
\bullet \ar@{<->}[r] & \bullet \ar@{<->}[r]  & \bullet \ar@{<->}[r]  & \bullet \ar@{<->}[r]  & \bullet
}}
\]
Observe that in each case the sources correspond to the vertices of the equipped graph, and there is a unique sink,
corresponding to the representation of the quiver which is 1-dimensional at each vertex.
The Auslander-Reiten quiver thus has a `rocket-shaped' appearance.

We now turn to the case of an extended Dynkin equipped graph~$E$.
We construct the preprojective component by knitting.
For the star-shaped graphs, $\tilde D_4$, $\tilde E_6$, $\tilde E_7$ and $\tilde E_8$,
the preprojective component contains all of the indecomposable representations of $E$ which are supported on one of
the arms of the graph. Thus the remaining indecomposable representations are all supported at the central vertex of the star.
Thus, considered as representations of the associated quiver $Q$, the outward pointing arrows are all surjective, and the inward pointing
arrow are all injective. Since we are only interested in representations in which all arrows are injective, it follows that the outward
pointing arrows must be isomorphisms. Thus, shrinking these arrows to a point, we are interested in representations of quiver $Q'$
with the same underlying graph as $E$, but oriented so that all arrows are pointing inwards. 
Now we only want the representations of $Q'$ in which these arrows are injective, which involves omitting a finite number of
preinjective representations of $Q'$. Thus for example, in Figure~\ref{f:Eeighttilde} we show the Auslander-Reiten quiver for $\tilde E_8$
\[
{\xymatrix@=1.5pc{
& & \bullet \ar@{<->}[d]  \\
\bullet \ar@{<->}[r] & \bullet \ar@{<->}[r]  & \bullet \ar@{<->}[r]  & \bullet \ar@{<->}[r] 
& \bullet \ar@{<->}[r] & \bullet \ar@{<->}[r]  & \bullet \ar@{<->}[r]  & \bullet
}}
\]
\begin{figure}
\[
\includegraphics[width=\textwidth]{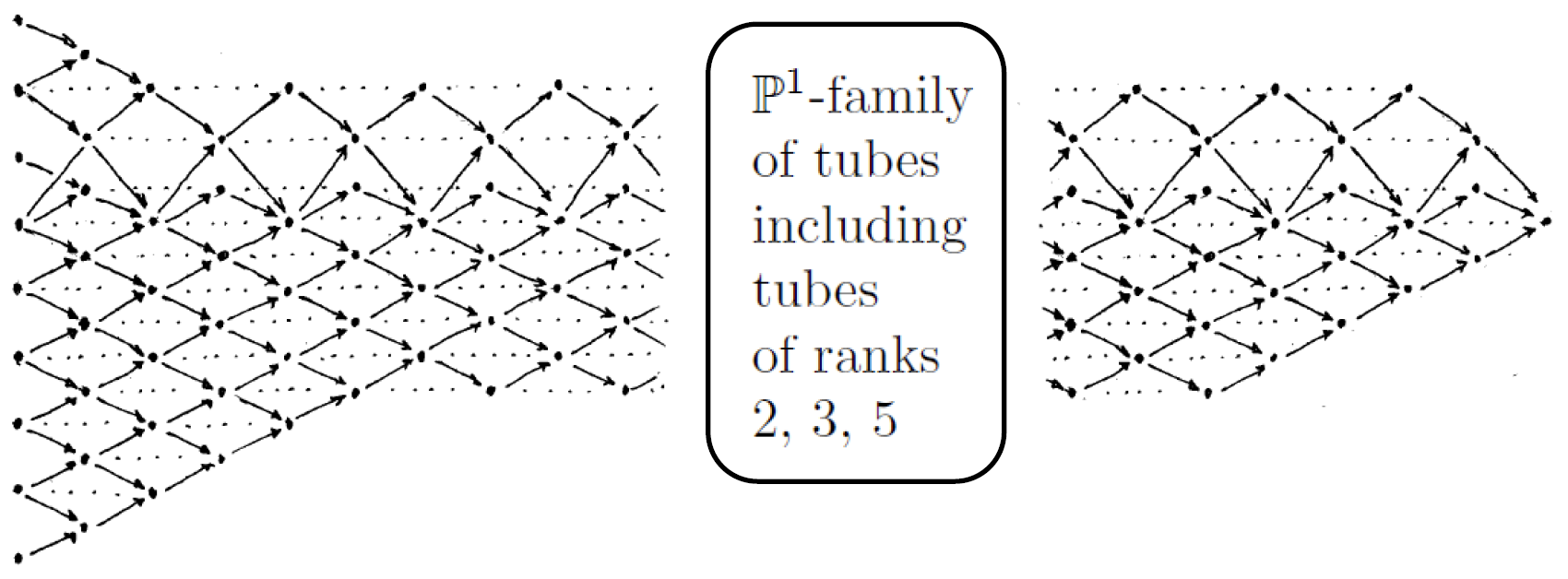}
\]
\vspace{-1cm}
\caption{Auslander-Reiten quiver for $\tilde E_8$}
\label{f:Eeighttilde}
\end{figure}

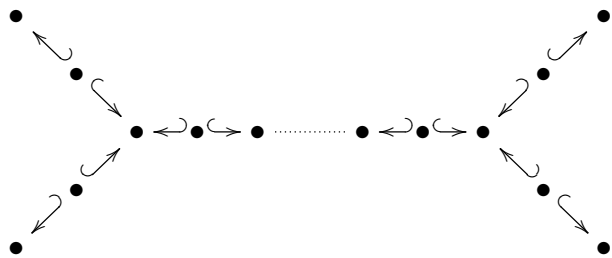
\begin{figure}
\[
{\xymatrix@=0.8pc{
\bullet & & & & & & & &  &  & \bullet  \\
& \bullet \ar@{_{(}->}[lu] \ar@{^{(}->}[rd] & & & & & & & & \bullet  \ar@{^{(}->}[ru] \ar@{_{(}->}[ld] \\
& & \bullet 
& \bullet \ar@{_{(}->}[l] \ar@{^{(}->}[r] & \bullet \ar@{.}[rr] & & \bullet 
& \bullet \ar@{_{(}->}[l] \ar@{^{(}->}[r] & 
\bullet  \\
& \bullet \ar@{_{(}->}[ld]   \ar@{^{(}->}[ru] & & & & & & & & \bullet  \ar@{^{(}->}[rd]  \ar@{_{(}->}[lu]\\
\bullet & & & & & & & & & & \bullet 
}}
\]
\vspace{-0.3cm}
\caption{The associated quiver for $E$ of type $\tilde D_n$}
\label{f:assocDntilde}
\end{figure}
\begin{figure}
\[
\includegraphics[width=\textwidth]{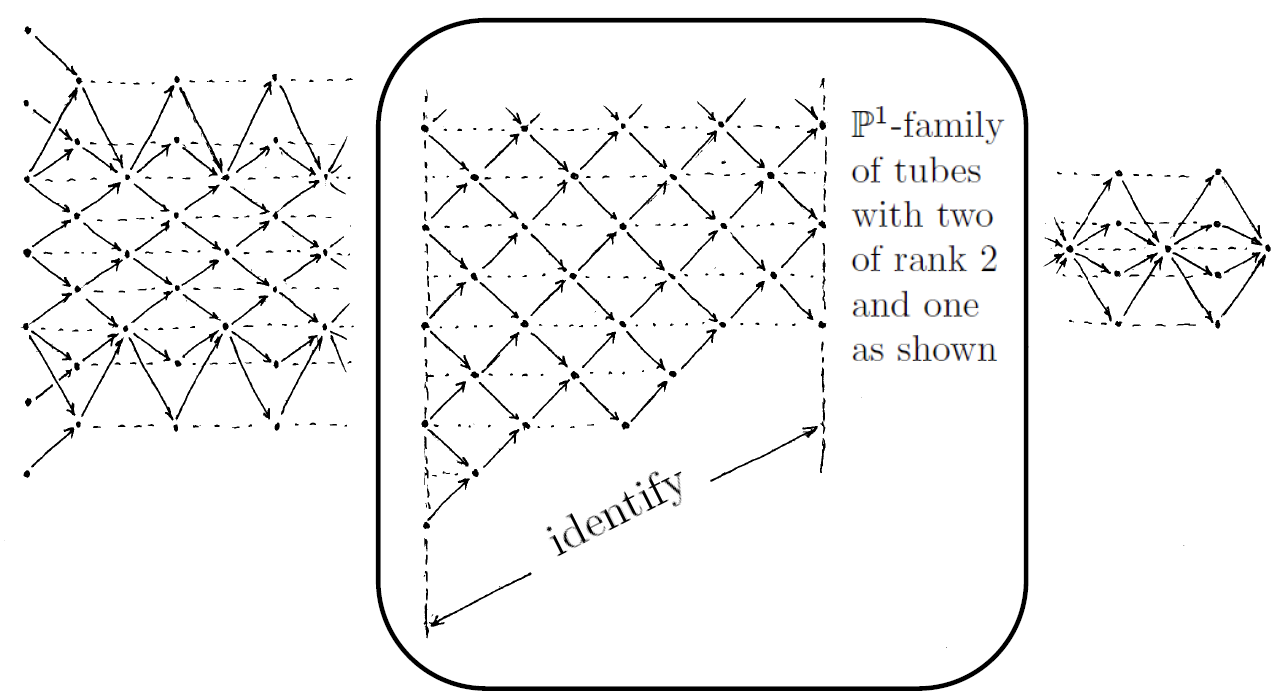}
\]
\caption{Auslander-Reiten quiver for $\tilde D_6$}
\label{f:Dsixtilde}
\end{figure}
For type $\tilde D_n$ with $n>4$, the associated quiver $Q$ has a central trunk and four arms of length 2, as in Figure~\ref{f:assocDntilde}.
The preprojective component contains all indecomposables supported on an arm, so the remaining indecomposable representations
of $Q$ are nonzero at at least one vertex on the trunk. It follows that the outward maps on the arms are surjective, and hence isomorphisms.
Thus these arrows can be shrunk to a point, giving a quiver $Q'$ of shape $\tilde D_{2n-4}$.
Using the classification of representations of $Q'$ one obtains the Auslander-Reiten quiver of $E$. 
We illustrate this in Figure~\ref{f:Dsixtilde} for the quiver $\tilde D_6$.

\begin{figure}
\[
\includegraphics[width=0.8\textwidth]{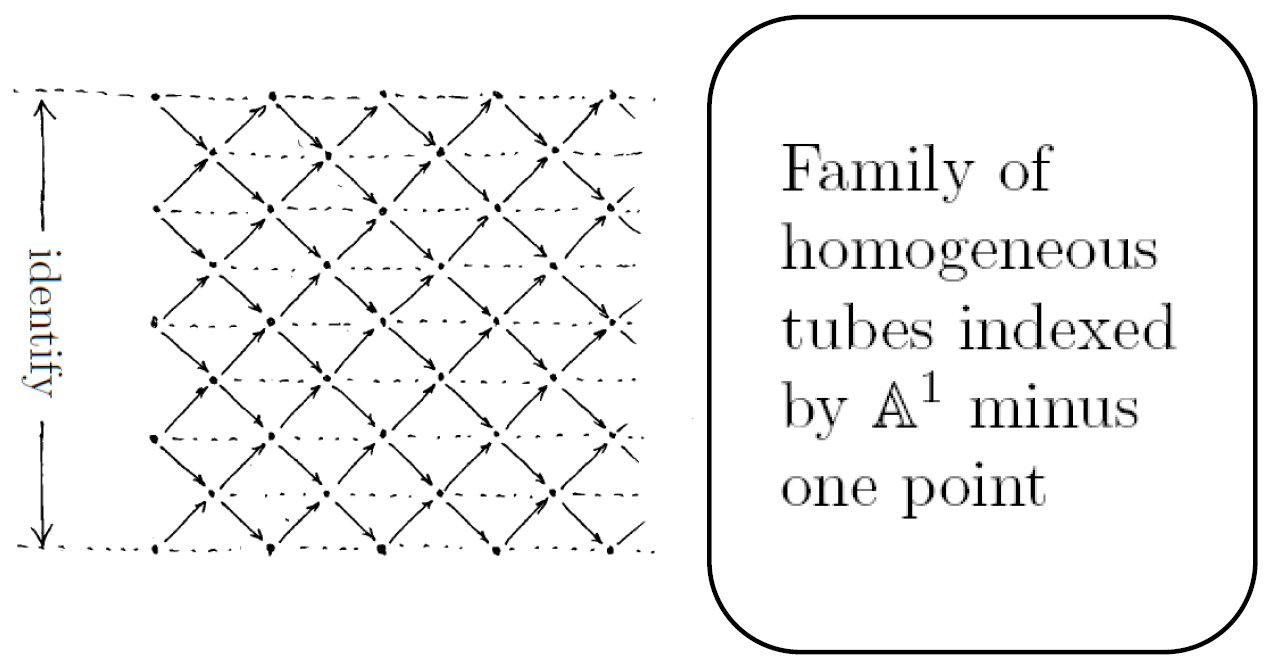}
\]
\vspace{-0.5cm}
\caption{Auslander-Reiten quiver for $\tilde A_3$}
\label{f:Athreetilde}
\end{figure}
Finally, although the extended Dynkin equipped graph $\tilde A_n$ does not meet our hypotheses, and is not functorially 
finite, 
it does still have an Auslander-Reiten quiver. 
In Figure~\ref{f:Athreetilde} this is shown for $\tilde A_3$ 
\[
{\xymatrix@=1.5pc{
\bullet \ar@{<->}[d] \ar@{<->}[r] & \bullet \ar@{<->}[d] \\
\bullet \ar@{<->}[r] & \bullet 
}}
\]
It is curious that it looks like the Auslander-Reiten quiver for $\tilde A_3$ as an oriented cycle quiver, but with one of the tubes on its side.

\end{document}